\documentclass{article}
\usepackage[latin1]{inputenc}
\usepackage{amsfonts}
\usepackage{amsmath, latexsym}
\usepackage{amssymb,verbatim}
\usepackage{mathtools}
\usepackage{enumitem}

\newenvironment{dedication}
     {\vspace{6ex}\begin{quotation}\begin{center}\begin{em}}
     {\par\end{em}\end{center}\end{quotation}}

\usepackage[nameinlink,capitalize]{cleveref}

\usepackage{graphicx, psfrag}
\usepackage{tikz}
\usetikzlibrary{arrows,3d,patterns}
\usepackage{color}

\usepackage{euscript}

\usepackage{amsthm}
\theoremstyle{plain}
\newtheorem{theorem}{Theorem}
\newtheorem{lemma}[theorem]{Lemma}
\newtheorem*{lemma*}{Lemma}

\theoremstyle{definition}

\newtheorem{remark}[theorem]{Remark}

\theoremstyle{remark}

\numberwithin{equation}{section}

\def\R{\mathbb{R}}

\def\Z{\mathbb{Z}}
\def\N{\mathbb{N}}

\def\T{\mathbb{T}}

\def\ve{\varepsilon}

\newcommand{\diff}{\mathop{}\mathopen{}\mathrm{d}}
\newcommand{\abs}[1]{\lvert#1 \rvert}

\newcommand\eps{\varepsilon}
\newcommand{\be}{\begin{equation}}
\newcommand{\ee}{\end{equation}}
\newcommand{\nd}{\noindent}

\newcommand{\card}{\textrm{card}}

\title{A necessary condition in a De Giorgi type conjecture for elliptic systems 
in infinite strips}

\author{
{\Large Radu Ignat}
\footnote{Institut de Math\'ematiques de Toulouse \& Institut Universitaire de France, UMR 5219, Universit\'e de Toulouse, CNRS, UPS
IMT, F-31062 Toulouse Cedex 9, France. Email: Radu.Ignat@math.univ-toulouse.fr} 
\and {\Large Antonin Monteil }\footnote{Institut de Recherche en Math\'ematique et Physique, Universit\'e catholique de Louvain, \'Ecole de Math\'ematique, Chemin du Cyclotron 2, bte L7.01.02, 1348 Louvain-la-Neuve, Belgium. Email: Antonin.Monteil@uclouvain.be
}
}

\begin{document}
\maketitle

\vspace{-1cm}
\begin{dedication}
Dedicated to Ha\"im Brezis on his seventy-fifth anniversary\\ with esteem
\end{dedication}

\bigskip

\begin{abstract}
Given a bounded Lipschitz domain \(\omega\subset\R^{d-1}\) and a lower semicontinuous function \(W:\R^N\to\R_+\cup\{+\infty\}\) that vanishes on a finite set and that is bounded from below by a positive constant at infinity, we show that every map \(u:\R\times\omega\to\R^N\) with 
$$\int_{\R\times\omega} \big(\abs{\nabla u}^2+W(u)\big)\diff x_1\diff x'<+\infty$$ has a limit \(u^\pm\in\{W=0\}\) as \(x_1\to\pm\infty\). The convergence holds in \(L^2(\omega)\) and almost everywhere in \(\omega\). We also prove a similar result for more general potentials $W$ in the case where the considered maps $u$ are divergence-free in $\Omega$ with $\omega$ being the $(d-1)$-torus and $N=d$.

\medskip

\nd {\bf Keywords.} Nonlinear elliptic PDEs; De Giorgi conjecture; Energy estimates; Geodesic distance.

\end{abstract}

\section{Introduction}

Let $N\geq 1$, $d\geq 2$ and $\Omega=\R\times \omega$ be an infinite cylinder in $\R^d$, where $\omega\subset \R^{d-1}$ is an open connected bounded set with Lipschitz boundary. For a lower semicontinuous potential $W:\R^N\to\R_+\cup\{+\infty\}$, 
we consider the functional 
\begin{equation}\label{EE}
E(u)= \int_\Omega \Big( |\nabla u|^2 + W(u) \Big) \diff x, \quad u\in\dot{H}^1(\Omega,\R^N),
\end{equation}
where $|\cdot|$ is the Euclidean norm 
and 
\begin{equation*}
\dot{H}^1(\Omega,\R^N)=\left\{u\in H^1_{loc}(\Omega,\R^N)\;:\;\nabla u=(\partial_j u_i)_{1\leq i\leq N, 1\leq j \leq d}\in L^2(\Omega,\R^{N\times d})\right\}.
\end{equation*}
A natural problem consists in studying optimal transition layers for the functional $E$ between two wells $u^\pm$ of $W$ (i.e., $W(u^\pm)=0$). In particular, motivated by the De Giorgi conjecture, one aim is to analyse under which conditions on the potential $W$ and on the dimensions $d$ and $N$, every minimizer $u$ of $E$ connecting $u^\pm$ as $x_1\to \pm \infty$ is one-dimensional, i.e., depending only on $x_1$. Obviously, such one-dimensional transition layers $u$ coincide with 
their $x'$-average $\overline{u}:\R\to \R^N$ defined as
\begin{equation}\label{average}
\overline{u}(x_1):=\int_{\omega} \hspace{-0.4cm} - \, \, u(x_1,x')\diff x', \quad x_1\in\R,
\end{equation}
where $x'=(x_2,\dots,x_d)$ denotes the $d-1$ variables in $\omega$ and the $x'$-average symbol is denoted by $\displaystyle \int_{\omega} \hspace{-0.4cm} - \, \,=\frac1{|\omega|} \int_{\omega}$.

\subsection{Main results}

The purpose of this note is to prove a necessary condition for finite energy configurations $u$ provided that $W$ satisfies the following two conditions:
\begin{description}
\item[\bf (H1)]
$W$ has a finite number of wells, i.e., $\card(\{z\in \R^N\, :\, W(z)=0\})<\infty$;
\item[\bf (H2)]
$\liminf\limits_{|z|\to\infty} W(z)>0$.
\end{description}
More precisely, we prove that under these assumptions, there exist two wells $u^\pm$ of $W$ such that $u(x_1, \cdot)$ converges to $u^\pm$ in $L^2$ and a.e. in $\omega$ as $x_1\to \pm\infty$; in particular, the  $x'$-average $\overline{u}$ (as a continuous map in $\R$) admits the limits  
$\overline{u}(\pm \infty)=u^\pm$ as $x_1\to \pm\infty$. Here, $u(x_1,\cdot)$ stands for the trace of the Sobolev map $u\in\dot{H}^1(\Omega,\R^N)$ on the section $\{x_1\}\times \omega$ for every $x_1\in \R$.

\medskip

\begin{theorem}\label{thm1}
Let $\Omega=\R\times \omega$, where $\omega\subset \R^{d-1}$ is an open connected bounded set with Lipschitz boundary. 
If $W:\R^N\to \R_+\cup\{+\infty\}$ is a lower semicontinuous potential satisfying {\bf (H1)} and {\bf (H2)}, then every $u \in\dot{H}^1(\Omega,\R^N)$ with $E(u)<\infty$ connects two wells\footnote{$u^-$ and $u^+$ could be equal.} $u^\pm\in\R^N$ of $W$ at $x_1=\pm \infty$ (i.e., $W(u^\pm)=0$) in the sense that
\begin{equation}\label{BC}
\lim_{x_1\to \pm \infty}\|u(x_1,\cdot)-u^\pm\|_{L^2(\omega,\R^N)}=0\quad\text{and}\quad
\lim_{x_1\to \pm \infty} u(x_1, \cdot)=u^\pm\quad\text{a.e. in } \omega.
\end{equation}
In particular,
$$
\lim\limits_{x_1\to \pm \infty} \int_{\omega} \hspace{-0.4cm} - \, \, u(x_1, x')\, \diff x'=u^\pm.
$$
\end{theorem}

\begin{remark}
i) As a consequence of the Poincar\'e-Wirtinger inequality\footnote{The assumption that $\omega$ is connected with Lipschitz boundary is needed for the Poincar\'e-Wirtinger inequality.}, for $u\in \dot{H}^1(\Omega, \R^N)$ with $\bar u(\pm \infty)=u^\pm$, there exist two sequences $(R_n^+)_{n\in\N}$ and $(R_n^-)_{n\in\N}$ such that $(R_n^\pm)_{n\in\N}\to\pm\infty$ and
\be
\label{poincare_wirtinger}
\| u(R_n^\pm,\cdot)- u^\pm\|_{H^1(\omega,\R^N)} \underset{n\to\infty}{\longrightarrow}0
\ee
(see \cite[Lemma~3.2]{Ignat-Monteil}). 

\medskip

ii) Theorem \ref{thm1} also holds true if $\omega$ is a closed (i.e., compact, connected without boundary) Riemannian manifold.

\medskip

iii) Theorem \ref{thm1} also applies for maps $u$ taking values into a closed set $\mathcal{N}\subset \R^N$ (e.g., $\mathcal{N}$ could be a compact manifold embedded in $\R^N$). More precisely, if the potential \(W:\R^N\to \R_+\cup\{+\infty\}\) satisfies {\bf (H1)}, {\bf (H2)} 
and \(\mathcal{N}:=\{z\in \R^N\, :\, W(z)<+\infty\}\) is a closed set such that \(W_{\vert \mathcal{N}}:\mathcal{N}\to\R_+\) is lower semicontinuous, then Theorem \ref{thm1} handles the case where the nonlinear constraint \(u\in\mathcal{N}\) is present.
\end{remark}

\medskip

The result in Theorem \ref{thm1} extends to slightly more general potentials $W$ in the following context of divergence-free maps. For that, let $d=N$ and $\Omega=\R\times \omega$ with $\omega=\mathbb{T}^{d-1}$ and $\T=\R/\Z$ being the flat torus.
We consider maps $u\in H^1_{loc}(\Omega, \R^d)$ periodic in $x'\in \omega$ and divergence-free, i.e., 
$$\nabla\cdot u=0 \quad \textrm{in} \quad \Omega.$$
Then the $x'$-average $\bar u:\R\to \R^d$ is continuous and its first component is constant, i.e., there is $a\in \R$ such that
$$\bar u_1(x_1)=a \quad \textrm{for every} \quad x_1\in \R$$
(see \cite[Lemma 3.1]{Ignat-Monteil}). For such maps $u$, we consider potentials $W$ satisfying the following two conditions:
\begin{description}
\item[$\textrm{\bf (H1)}_a$]
$W(a, \cdot)$ has a finite number of wells, i.e., $\card(\{z'\in \R^{d-1}\, :\, W(a,z')=0\})<\infty$;
\item[$\textrm{\bf (H2)}_a$]
$\liminf\limits_{z_1\to a, \, |z'|\to\infty} W(z_1, z')>0$.
\end{description}
In this context, we have proved in our previous paper \cite{Ignat-Monteil} that the $x'$-average map $\bar u$ admits limits $u^\pm$ as $x_1\to \pm \infty$, where $u^\pm_1=a$ and they are two wells of $W(a, \cdot)$, see \cite[Lemma 3.7]{Ignat-Monteil}. As in Theorem \ref{thm1}, we will prove that $u(x_1, \cdot)$ converges to $u^\pm$ in $L^2$ and a.e. in $\omega$ as $x_1\to \pm\infty$.

\medskip

\begin{theorem}\label{thm2}
Let $\Omega=\R\times \omega$ with $\omega=\T^{d-1}$ the $(d-1)$-dimensional torus and $u\in H^1_{loc}(\Omega, \R^d)$ such that $E(u)<\infty$ and $\bar u_1=a$ in $\R$ for some $a\in \R$. 
If $W:\R^d\to \R_+\cup\{+\infty\}$ is a lower semicontinuous potential satisfying $\textrm{\bf (H1)}_a$ and 
$\textrm{\bf (H2)}_a$, then there exist two wells $u^\pm\in\R^d$ of $W$ such that \eqref{BC}
holds true and $u^\pm_1=a$. In particular, $\bar u(\pm\infty)=u^\pm$.
\end{theorem}

Note that we don't assume that $u$ is divergence-free in Theorem \ref{thm2}, only the assumption that
$\bar u_1$ is constant.

\subsection{Motivation}

Our main result is motivated by the well-known De Giorgi conjecture that consists in investigating 
the one-dimensional symmetry of critical points of the functional $E$, i.e., solutions $u:\Omega\to\R^N$ to the nonlinear elliptic system
\begin{equation}
\label{sys}
\begin{cases}
\Delta u=\frac12 \nabla W(u) \quad & \text{in} \quad \Omega,\\
\frac{\partial u}{\partial \nu}=0 \quad & \text{on} \quad \partial\Omega=\R\times \partial \omega,
\end{cases}
\end{equation}
where $W$ is assumed to be locally Lipschitz in \eqref{sys} and $\nu$ is the unit outer normal vector field at $\partial \omega$.
Theorem \ref{thm1} states in particular that solutions $u$ of finite energy  satisfy the boundary condition \eqref{BC} for two wells $u^\pm$ of $W$.
A natural question related to the De Giorgi conjecture arises in this context:

\medskip

\noindent {\bf Question}: Under which assumptions on the potential $W$ and the dimensions $d$ and $N$,
is it true that every global minimizer $u$ of $E$ connecting two wells\footnote{We say that $u$ connects two wells $u^\pm$ of $W$ if \eqref{BC} is satisfied.} of $W$  is one-dimensional symmetric, i.e., $u=u(x_1)$ ?

\medskip

\noindent {\it Link with the Gibbons and De Giorgi conjectures.} i) In the scalar case $N=1$ ($d$ is arbitrary) and $W(u)=\frac12(1-u^2)^2$, the answer to the above question is positive provided that the limits \eqref{BC} are replaced by uniform convergence (see \cite{Carbou,F1}); within these uniform boundary conditions, the problem is called the Gibbons conjecture. We mention that many articles have been written on Gibbons' conjecture in the case of the entire space $\Omega=\R^d$: more precisely, if a solution\footnote{Here, $u$ needs not be a global minimizer of $E$ within the boundary condition \eqref{BC}, nor monotone in $x_1$, i.e., $\partial_1 u>0$. Obviously, this result applies also to global minimizers, as $|u|\le 1$ in $\R^d$ by the maximum principle.} $u:\R^d\to \R$ of the PDE 
\be
\label{scalarPDE}
\Delta u=\frac12\frac{dW}{du}(u) \quad \textrm{ in }\quad \R^d
\ee
satisfies the convergence $\lim_{x_1\to\pm\infty}u(x_1,x')=\pm 1$ uniformly in $x'\in \R^{d-1}$ and \(|u|\le 1\) in $\R^d$,
then $u$ is one-dimensional (see \cite{BBG,BHM,CafCor,F2}).

Let us now speak about the long standing De Giorgi conjecture in the scalar case $N=1$. It predicts that any bounded solution $u$ of \eqref{scalarPDE} that is monotone in the $x_1$ variable is one-dimensional in dimension $d\leq 8$, i.e., 
the level sets $\{u=\lambda\}$ of $u$ are hyperplanes. The conjecture has been solved in dimension $d=2$ by Ghoussoub-Gui \cite{Ghoussoub:1998}, using a Liouville-type theorem and monotonicity formulas. Using similar techniques, Ambrosio-Cabr\'e \cite{AmbrosioCabre:2000} extended these results to dimension $d=3$, while Ghoussoub-Gui \cite{ghoussoub2003giorgi} showed that the conjecture is true for $d=4$ and $d=5$ under some antisymmetry condition on $u$. The conjecture was finally proved by Savin  \cite{Savin:2009} in dimension $d\leq 8$ under the additional condition $\lim_{x_1\to\pm\infty}u(x_1,x')=\pm 1$ pointwise in $x'\in\R^{d-1}$, the proof being based on fine regularity results on the level sets of $u$. Lately, Del Pino-Kowalczyk-Wei  \cite{del2011giorgi} gave a counterexample to the De Giorgi conjecture in dimension $d\geq 9$, which satisfies the pointwise limit conditions $\lim_{x_1\to\pm\infty}u(x_1,x')=\pm1$ for a.e. \(x'\in\R^{d-1}\). It would be interesting to investigate whether these results transfer (or not) to the context of the strip $\Omega=\R\times \omega$ as stated in Question. Theorem \ref{thm1} proves that the pointwise convergence as $x_1\to \pm \infty$ is a necessary condition in the context of a strip $\R\times \omega$ and for finite energy configurations.

\medskip

\noindent ii) Less results are available for the vector-valued case $N\geq 2$. In the case $\Omega=\R^d$, $N=2$ and 
$W(u_1, u_2)=\frac12(u_1^2-1)^2+\frac12(u_2^2-1)^2+{\Lambda}u_1^2 u_2^2-\frac12$ with $\Lambda\geq 1$
(so $W\geq 0$ and $W$ has exactly four wells $\{(0, \pm 1), (\pm 1, 0)\}$, thus, {\rm \bf (H1)} and {\rm \bf (H2)} are satisfied), the Gibbons and De Giorgi conjectures corresponding to the system \eqref{sys} are discussed in \cite{FSS}. Several other phase separation models (e.g., arising in a binary mixture of Bose-Einstein condensates) are studied in the vectorial case where $W$ has a non-discrete set of zeros (see e.g., \cite{berestycki2013phase, berestycki2013entire, fazly2013giorgi}). 

\medskip

We recall that in the study of the De Giorgi conjecture for \eqref{scalarPDE}, i.e., $N=1$, there is a link between monotonicity of solutions (e.g., the condition $\partial_1 u>0$), stability (i.e., the second variation of the corresponding energy at $u$ is nonnegative), and local minimality of $u$ (in the sense that the energy does not decrease under compactly supported perturbations of $u$). We refer to \cite[Section~4]{Alberti:2001} for a fine study of these properties. In particular, it is shown that the monotonicity condition in the De Giorgi conjecture implies that $u$ is a local minimizer of the energy (see \cite[Theorem~4.4]{Alberti:2001}). Therefore, it is natural to study Question under the monotonicity condition in $x_1$ (instead of the global minimality condition on $u$).

\bigskip

\noindent {\it Link with micromagnetic models.}
We have studied Question in the context of divergence-free maps $u:\R\times \omega\to \R^N$ where $d=N$ and $\omega=\mathbb{T}^{d-1}$ is the $(d-1)$-dimensional torus, see \cite{Ignat-Monteil}. By developing a theory of calibrations, we have succeeded to give sufficient conditions on the potential  $W$ in order that the answer to Question is positive, in particular in the case where $\textrm{\bf (H1)}_a$  and $\textrm{\bf (H2)}_a$  are satisfied, see \cite[Theorem 2.11]{Ignat-Monteil}. In that context, Question is related to some reduced model in micromagnetics in the regime where the so-called stray-field energy is strongly penalized favoring the divergence constraint $\nabla \cdot u=0$ of the magnetization $u$ (the unit-length constraint on $u$ being relaxed in the system). In the theory of micromagnetics, a challenging question concerns the symmetry of domain walls. Indeed, much effort has been devoted lately to identifying on the one hand, the domain walls that have one-dimensional symmetry, such as the so-called symmetric N\'eel and symmetric Bloch walls (see e.g. \cite{DKO, IO, Ignat:2011}), and on the other hand, the domain walls involving microstructures, such as the so-called cross-tie walls (see e.g., \cite{Alouges:2002,Riviere:2001}), the zigzag walls (see e.g., \cite{Ignat:2012, Moser_zigzag}) or the asymmetric N\'eel / Bloch walls 
(see e.g. \cite{Doring:2013, DI}). Thus, answering to Question would give a general approach in identifying the anisotropy potentials $W$ for which the domain walls are one-dimensional in the elliptic system \eqref{sys}.

\bigskip
\noindent {\it Link with heteroclinic connections.} One dimensional \footnote{If \(u=u(x_1)\), the Neumann condition \(\frac{\partial u}{\partial \nu}=0\) is automatically satisfied.} solutions \(u=u(x_1)\) of the system \eqref{sys} are called heteroclinic connections. Given two wells \(u^\pm\) of a potential \(W\) satisfying {\rm\bf (H1)} and {\rm\bf (H2)}, it is known that there exists a heteroclinic connection \(\gamma:\R\to\R^N\) obtained by minimizing 
\(\int_\R \abs{\frac{d}{dx_1}\gamma}^2+W(\gamma)\, dx_1\) under the condition \(\gamma(\pm\infty)=u^\pm\) (see \cite{MonteilSantambrogio2016,Sourdis:2016,Zuniga:2016}). In the vectorial case \(N\ge 2\), this connection may not be unique in the sense that there could exist two (minimizing) heteroclinic connections \(\gamma_1,\gamma_2\) such that \(\gamma_i(\pm\infty)=u^\pm\) for \(i=1,2\) but  \(\gamma_1(\cdot)\) and \(\gamma_2(\cdot-\tau)\) are distinct for every \(\tau\in\R\). If this is the case, at least in dimension \(d=2\) and $\Omega=\R^2$, there also exists a solution \(u\) to \(\Delta u=\frac 12 \nabla W(u)\) which realizes an interpolation between \(\gamma_1\) and \(\gamma_2\) in the following sense (see \cite{Schatzman:2002,Alama:1997,MonteilSantambrogio2017}):
\[
\begin{cases}
u(x_1,x_2)\to u^\pm&\text{as \(x_1\to\pm\infty\) uniformly in \(x_2\),}\\
u(x_1,x_2)\to \gamma_1(x_1)&\text{as \(x_2\to -\infty\) uniformly in \(x_1\),}\\
u(x_1,x_2)\to \gamma_2(x_1)&\text{as \(x_2\to +\infty\) uniformly in \(x_1\).}
\end{cases}
\]
Moreover, this solution is energy local minimizing, i.e., the energy cannot decrease by compactly supported perturbations of $u$. Solutions to the system \(\Delta u=\frac 12 \nabla W(u)\) naturally arise when looking at the local behavior of a transition layer near a point at the interface between two wells \(u^\pm\) ; solutions satisfying the preceding boundary conditions correspond to the case of an interface point where the 1D connection passes from \(\gamma_1\) to \(\gamma_2\). The existence of such stable entire solutions to the Allen-Cahn system makes a significative difference with the scalar case, i.e. \(N=1\), where only 1D solutions are present by the De Giorgi conjecture.

\section{Pointwise convergence and convergence of the $x'$-average}
\label{sec:2}

In this section we prove that under the assumptions in Theorem \ref{thm1},
the $x'$-average $\overline{u}$ (as a continuous map in $\R$) has limits $\overline{u}(\pm \infty)=u^\pm$ as $x_1\to \pm \infty$ corresponding to two wells of $W$. For that, we will follow the strategy that we developed in our previous paper (see \cite[Section 3.1]{Ignat-Monteil}). The idea consists 
in introducing an ``averaged" potential \(V\) in $\R^N$ with $W\geq V\geq 0$ and $\{V=0\}=\{W=0\}$ (see Lemma \ref{lemmaV}), and a new functional $E_V$ associated to the  $x'$-average $\overline{u}$ of a map $u$ such that $\frac1{|\omega|}E(u)\geq E_V(\bar u)$. This can be seen as a dimension reduction technique since the new map $\bar u$ has only one variable. We will prove that every transition layer $\bar u$ connecting two wells $u^\pm$ has the energy $E_V(\bar u)$ bounded from below by the geodesic pseudo-distance $\mathrm{geod}_V$ between the wells $u^\pm$
(see Lemma \ref{regularization_lemma}). As the Euclidean distance in $\R^N$ is absolutely continuous with respect to $\mathrm{geod}_V$ (see Lemma \ref{cWgeqGeod}), we will conclude that $\bar u$ admits limits at $\pm \infty$ given by two wells of $W$ (see Lemma \ref{closed_boundary}). Note that in Section \ref{sec3}, we will give a second proof of the claim $\bar u(\pm\infty)=u^\pm$ without using the geodesic pseudo-distance $\mathrm{geod}_V$.

We first introduce the energy functional $E$ (defined in \eqref{EE}) restricted to appropriate subsets $A\subset\Omega$ (e.g., $A$ can be a subset of the form $I\times \omega$ for an interval $I\subset \R$, or a section $\{x_1\}\times \omega$): for every map $u\in \dot{H}^1(A,\R^N)$, we set
\[
E(u,A):= 
 \int_A |\nabla u|^2 + W(u) \diff x,
\]
so that for $A=\Omega$, we have $E(u)=E(u, A)$. For any interval $I\subset\R$, the Jensen inequality yields
\[
E(u,I\times \omega)= \int_{I}\int_{\omega}\left( |\partial_{1}u|^2+|\nabla' u|^2+W(u)\right)\diff x'\diff x_1\geq 
 |\omega| \int_I \Big|\frac{\diff}{\diff x_1}\overline{u}(x_1)\Big|^2+e(u(x_1,\cdot))\diff x_1,
\]
where $\nabla'=(\partial_2, \dots, \partial_d)$, $\bar u$ is the $x'$-average of $u$ given in \eqref{average} and the $x'$-average energy $e$ is defined by 
$$e(v):= \int_{\omega} \hspace{-0.4cm} - \, \,\left( |\nabla' v|^2+W(v)\right)\diff x' \quad \textrm{ for all }\, v\in H^1(\omega,\R^N).$$ 
Introducing the averaged potential $V:\R^N\to\R_+\cup\{+\infty\}$ defined for all $z\in\R^N$ by 
\be
\label{defV}
V(z):=\inf\left\{ e(v)\;:\; v\in H^1(\omega,\R^N),\, \int_{\omega} \hspace{-0.4cm} - \, \, v\diff x'=z \right\}\geq 0,
\ee
we have
\be
\label{inegEV}
E(u,I\times \omega)\geq |\omega| \int_I \left(\Big|\frac{\diff}{\diff x_1}\overline{u}(x_1)\Big|^2+V(\overline{u}(x_1))\right)\diff x_1.
\ee
This observation is the starting point in the proof of the following lemma:
\begin{lemma}\label{lemmaV}
Let $W:\R^N\to\R_+\cup\{+\infty\}$ be a lower semicontinuous function satisfying {\rm \bf(H2)}. Then 
the averaged potential $V:\R^N\to\R_+\cup\{+\infty\}$ defined in \eqref{defV} satisfies the following:
\begin{enumerate}
\item
$V$ is lower semicontinuous in $\R^N$,
\item
for all $z\in\R^N$, $V(z)\leq W(z)$, the infimum in \eqref{defV} is achieved and\footnote{In particular, if $W$ satisfies {\rm \bf(H1)}, then $V$ satisfies {\rm \bf(H1)}, too.} 
$\Big[V(z)=0\Leftrightarrow W(z)=0\Big]$,
\item
$V_\infty:=\liminf\limits_{|z|\to\infty}V(z)>0$,
\item
for every interval $I\subset\R$ and for every $u\in \dot{H}^1(I\times\omega,\R^N)$, one has
\[
\frac1{|\omega|}E(u,I\times \omega)\geq  E_V(\overline{u},I), \quad E_V(\overline{u},I) :=\int_I \Big|\frac{\diff}{\diff x_1}\overline{u}(x_1)\Big|^2+V(\overline{u}(x_1))\diff x_1.
\]
\end{enumerate}
\end{lemma}

The new energy $E_V(\bar u):=E_V(\overline{u}, \R)$ associated to the $x'$-average $\bar u$ will play an important role for proving the existence of the two limits $\bar u(\pm \infty)$.

\begin{proof}[Proof of Lemma \ref{lemmaV}] The claim {\it 4} follows from \eqref{inegEV}. We divide the rest of the proof in three steps. 

\medskip
\noindent\textsc{Step 1: proof of claim {\it 2}.} Clearly, for all $z\in\R^N$, one has $V(z)\leq e(z)=W(z)$. By the compact embedding $H^1(\omega)\hookrightarrow L^1(\omega)$, the lower semicontinuity of $W$, Fatou's lemma and the lower semicontinuity of the $L^2$ norm in the weak $L^2$-topology (see \cite{Brezis}), we deduce that $e$ is lower semicontinuous in the weak $H^1(\omega,\R^N)$-topology. Then the direct method in the calculus of variations implies that the infimum is achieved in \eqref{defV} (infimum that could be equal to $+\infty$ as $W$ can take the value $+\infty$). 

If $W(z)=0$, then $V(z)=0$ (as $0\leq V\leq W$ in $\R^N$). 
Conversely, if $V(z)=0$ with $z\in\R^N$, then a minimizer $v\in H^1(\omega,\R^N)$ in \eqref{defV} satisfies $V(z)=e(v)=0$ so that $v\equiv z$ and $W(z)=0$.

\medskip
\noindent\textsc{Step 2: $V$ is lower semicontinuous in $\R^N$.} Let $(z_n)_{n\in\N}$ be a sequence converging to $z$ in $\R^N$. We need to show that
\[
V(z)\leq\liminf_{n\to\infty}V(z_n).
\] 
Without loss of generality, one can assume that $(V(z_n))_{n\in\N}$ is a bounded sequence that converges to $\liminf_{n\to\infty} V(z_n)$. By Step 1, for each $n\in\N$, there exists $v_n\in H^1(\omega,\R^N)$ such that 
\[
 \int_{\omega} \hspace{-0.4cm} - \, \,  v_n \diff x'=z_n\quad\text{and}\quad e(v_n)=V(z_n).
\]
Since $(z_n)_{n\in\N}$ and $(e(v_n))_{n\in\N}$ are bounded, we deduce that $(v_n)_{n\in\N}$ is bounded in $H^1(\omega,\R^N)$ by the Poincar\'e-Wirtinger inequality. Thus, up to extraction, one can assume that $(v_n)_{n\in\N}$ converges weakly in $H^1$, strongly in $L^1$ and a.e. in $\omega$ to a limit 
$v\in H^1(\omega,\R^N)$. In particular, $  \int_{\omega} \hspace{-0.45cm} - \, \,  v \diff x'=z$. Since $e$ is lower semicontinuous in weak $H^1(\omega,\R^N)$-topology (by Step 1), we conclude
\[
V(z)\leq e(v)\leq \liminf\limits_{n\to\infty} e(v_n)= \liminf\limits_{n\to\infty} V(z_n).
\]
\noindent\textsc{Step 3: proof of claim {\it 3}.} Assume by contradiction that there exists a sequence $(z_n)_{n\in\N}\subset\R^N$ such that $|z_n|\to\infty$ and $V(z_n)\to 0$ as $n\to\infty$. Then, there exists a sequence of maps $(w_n)_{n\in\N}$ in $H^1(\omega,\R^N)$ satisfying
\[
\int_{\omega}w_n(x')\diff x'=0\quad\text{for each $n\in\N$}\quad\text{and}\quad e(z_n+w_n)\underset{n\to\infty}{\longrightarrow} 0.
\]
By the Poincar\'e-Wirtinger inequality, we have that $(w_n)_{n\in\N}$ is bounded in $H^1$. Thus, up to extraction, one can assume that it converges weakly in $H^1$, strongly in $L^1$ and a.e. to a map $w\in H^1(\omega,\R^N)$. We claim that $w$ is constant since
$$
 \int_{\omega} \hspace{-0.4cm} - \, \,  | \nabla' w|^2\diff x'\le \liminf\limits_{n\to\infty} \int_{\omega} \hspace{-0.4cm} - \, \,  | \nabla' w_n|^2\diff x'\le \liminf\limits_{n\to\infty}e(z_n+w_n)=0.
$$
We deduce $w\equiv 0$ since $\int_{\omega} w=\lim_{n\to\infty}\int_{\omega} w_n=0$. Thus $w_n\to 0$ a.e and 
{\rm \bf{(H2)}} implies that for a.e.\ $x'\in \omega$,
$$\liminf_{n\to\infty}W(z_n+w_n(x'))\geq \liminf\limits_{|z|\to\infty} W(z)>0,$$
which contradicts the fact that $e(z_n+w_n)\to 0$.
\end{proof}

\bigskip

For every lower semicontinuous function $W:\R^N\to\R_+\cup\{+\infty\}$ satisfying {\rm \bf{(H1)}} and {\rm \bf{(H2)}}, we introduce the geodesic pseudo-distance $\mathrm{geod}_W$ in $\R^N$ endowed with the singular pseudo-metric $4Wg_0$, $g_0$ being the standard Euclidean metric in $\R^N$; this geodesic pseudo-distance (that can take the value $+\infty$) is defined for every $ x,y\in\R^N$ by
\begin{multline}
\label{estgeod}
\mathrm{geod}_W(x,y):=\inf\bigg\{\int^1_{ -1} 2\sqrt{W(\sigma(t))}|\dot{\sigma}|(t)\diff t\;:\; \sigma\in\mathrm{Lip}_{ploc}([-1,1],\R^N),\, \sigma( -1)=x,\,\sigma (1)=y\bigg\},
\end{multline}
where $\mathrm{Lip}_{ploc}([-1,1],\R^N)$ is the set of continuous and {\bf piecewise locally Lipschitz} curves
\footnote{In general, we cannot hope that a minimizing sequence in \eqref{estgeod} is better than piecewise locally Lipschitz because $W$ is not assumed locally bounded ($\dot \sigma$ is the derivative of $\sigma$). However, in the case of a locally bounded $W$, we could use a regularization procedure in order to restrict to Lipschitz curves $\sigma$.} 
on $[-1,1]$:
\begin{align*}
\mathrm{Lip}_{ploc}([-1,1],\R^N):=\Big\{ \sigma\in \mathcal{C}^0([-1,1],\R^N)\;:\; & \textrm{there is a partition } -1= t_1<\dots< t_{k+1}= 1,\\
& \text{with } \sigma\in \mathrm{Lip}_{loc}((t_i,t_{i+1})) \,  \textrm{ for every } 1\leq i\leq k\Big\}.
\end{align*}
By {\it pseudo-distance}, we mean that $\mathrm{geod}_W$ satisfies all the axioms of a distance; the only difference with respect to the standard definition is that a pseudo-distance can take the value $+\infty$. We will prove that $\mathrm{geod}_W$ yields a lower bound for the energy $E$ (see Lemma \ref{regularization_lemma}); this plays an important role in the proof of our claim $\overline{u}(\pm \infty)=u^\pm$. 

We start by proving some elementary facts about the pseudo-metric structure induced by $\mathrm{geod}_W$ on $\R^N$:
\medskip
 
\begin{lemma}\label{cWgeqGeod}
Let $W:\R^N\to\R_+\cup\{+\infty\}$ be a lower semicontinuous function satisfying {\rm \bf (H1)} and {\rm \bf(H2)}. Then the function $\mathrm{geod}_W:\R^N\times\R^N\to\R_+\cup\{+\infty\}$ defines a pseudo-distance over $\R^N$ and the Euclidean distance is absolutely continuous with respect to 
$\mathrm{geod}_W$, i.e., for every $\delta>0$, there exists $\varepsilon>0$ such that for every $x,y\in \R^N$ with $\mathrm{geod}_W(x,y)< \varepsilon$, we have $|x-y|< \delta$. 
\end{lemma}

\begin{proof}[Proof of Lemma \ref{cWgeqGeod}] In proving that $\mathrm{geod}_W:\R^N\times\R^N\to\R_+\cup\{+\infty\}$ defines a pseudo-distance over $\R^N$, the only non-trivial axiom to check is the non-degeneracy, i.e., $\mathrm{geod}_W(x,y)>0$ whenever $x\neq y$. In fact, we prove the stronger property that for every $\delta>0$, there exists $\varepsilon>0$ such that for every $x,y\in \R^N$, $|x-y|\ge\delta$ implies $\mathrm{geod}_W(x,y)\ge\varepsilon$ which also yields the absolute continuity of the Euclidean distance with respect to 
$\mathrm{geod}_W$. For that, we recall that the set $\{W=0\}$ is finite (by {\rm \bf(H1)}); therefore, w.l.o.g. we can assume that $\delta>0$ is small enough so that the open balls $B(p,\delta/2)$, for $p\in\{W=0\}$, are disjoint. We consider the following disjoint union of balls
\[
\Sigma_\delta:=\bigsqcup_{p\in\{W=0\}} B(p, \frac\delta 4),
\]
the distance between each ball being larger than $\delta/2$. We now take two points $x,y\in\R^N$ with $|x-y|\ge \delta$. In order to obtain a lower bound on $\mathrm{geod}_W(x,y)$, we take an arbitrary continuous and piecewise locally Lipschitz curve $\sigma:[-1,1]\to\R^N$ such that $\sigma(-1)=x$ and $\sigma(1)=y$. As $\abs{x-y}\ge\delta$ (so no ball in $\Sigma_\delta$ can contain both $x$ and $y$), by connectedness, the image $\sigma([-1,1])$ cannot be contained in $\Sigma_\delta$. Thus, there exists $t_0\in [-1,1]$ with $\sigma(t_0)\notin \Sigma_\delta$. It implies that $B(\sigma(t_0),\delta/8) \cap \Sigma_{\delta/2}=\emptyset$. Moreover, since $\abs{x-y}\ge\delta$, we have either $\abs{\sigma(t_0)-x}\ge \delta/2$ or $\abs{\sigma(t_0)-y}\ge \delta/2$; w.l.o.g., we may assume that $\abs{\sigma(t_0)-y}\ge \delta/2$. Then the (continuous) curve $\sigma\big|_{[t_0,1]}$ has to get out of the ball $B(\sigma(t_0),\delta/8)$; in particular, it has length larger than $\delta/8$ and 
\[
\int_{-1}^1 2\sqrt{W(\sigma(t))}|\dot{\sigma}|(t)\diff t\ge\frac\delta{4}  \inf_{z\in B(\sigma(t_0),\delta/8)}\sqrt{W(z)}\ge \frac\delta {4}  \inf_{z\in \R^N\setminus\Sigma_{\delta/2}}\sqrt{W(z)}.
\]
Since $W$ is lower semicontinuous and bounded from below at infinity (by {\rm\bf (H2)}), we deduce that $W$ is bounded from below by a constant $c_\delta>0$ on $\R^N\setminus\Sigma_{\delta/2}$. Taking the infimum over curves $\sigma\in \mathrm{Lip}_{ploc}([-1,1],\R^N)$ connecting $x$ to $y$, we deduce from the preceding lower bound that
\[
\mathrm{geod}_W(x,y)\ge \frac{\delta\sqrt{c_\delta}}{4}>0.
\]
This finishes the proof of the result.
\end{proof}

We now use a regularization argument to derive the following lower bound on the energy:
\begin{lemma}
\label{regularization_lemma}
Let $W:\R^N\to\R_+\cup\{+\infty\}$ be a lower semicontinuous function. Then, for every interval $I\subset \R$ and every map $\sigma\in \dot{H}^1(I,\R^N)$ having limits $\sigma(\inf I)$ and $\sigma(\sup I)$ at the endpoints of $I$, we have
\be
\label{ineg_ener}
E_W(\sigma, I):=\int_I \Big(\abs{\dot\sigma(t)}^2+W(\sigma(t))\Big)\diff t\ge \mathrm{geod}_W\big(\sigma(\inf I),\sigma(\sup I)\big).
\ee
\end{lemma}

\begin{proof}[Proof of Lemma \ref{regularization_lemma}]
W.l.o.g. we assume that $I$ is an open interval. Since $\dot{H}^1(I,\R^N)\subset W^{1,1}_{loc}(I, \R^N)$, we can define the arc-length $s:I\to J:=s(I)\subset\R$ by
\[
s(t):=\int_{t_0}^t|\dot{\sigma}|(x_1)\diff x_1,\quad t\in I,
\]
where $t_0\in I$ is fixed. Thus $s$ is a nondecreasing continuous function with $\dot{s}=|\dot{\sigma}|$ a.e. in $I$. Then the arc-length reparametrization of $\sigma$, i.e.
\[
\tilde{\sigma}(s(t)):=\sigma(t),\quad t\in I,
\]
is well-defined and provides a Lipschitz curve $\tilde{\sigma}:J\to\R^N$ with constant speed on the interval $J$, i.e. $|\dot{\tilde\sigma}|=1$ a.e., and such that $\tilde{\sigma}(\inf J)={\sigma}(\inf I)$ and $\tilde{\sigma}(\sup J)={\sigma}(\sup I)$. 
W.l.o.g. we may assume that  $\sigma$ is not constant, so $J$ has a nonempty interior. Then we consider an arbitrary function $\varphi\in\mathrm{Lip}_{loc}((-1,1),\mathrm{int}J)$ which is nondecreasing and surjective onto the interior of the interval $J$ and we set
\[
\gamma(t):=\tilde{\sigma}(\varphi(t)),\quad t\in (-1,1).
\]
So $\gamma$ is a locally Lipschitz map that is continuous on $[-1,1]$ as $\tilde{\sigma}$ admits limits at $\inf J$ and 
$\sup J$; thus, $\gamma\in \mathrm{Lip}_{ploc}([-1,1],\R^N)$. The changes of variable $s:=\varphi(t)$, resp. $s:=s(t)$, yield
\[
\int_{-1}^1 2\sqrt{W(\gamma(t))}|\dot{\gamma}|(t)\diff t=\int_J 2\sqrt{W(\tilde{\sigma}(s))}|\dot{\tilde{\sigma}}|(s)\diff s=\int_I 2\sqrt{W(\sigma(t))}\, |\dot{\sigma}|(t)\diff t.
\]
Combined with $\gamma(-1)=\sigma(\inf I)$ and $\gamma(1)=\sigma(\sup I)$, the definition of $\mathrm{geod}_W$ and the Young inequality imply
\[
E_W(\sigma, I)\ge \int_I 2\sqrt{W(\sigma(t))}\, |\dot{\sigma}|(t)\diff t=\int_{-1}^1 2\sqrt{W(\gamma(t))}|\dot{\gamma}|(t)\diff t\ge  \mathrm{geod}_W\big(\sigma(\inf I),\sigma(\sup I)\big).
\]
This completes the proof.
\end{proof}

The convergence of the $x'$-average in Theorem \ref{thm1} stating that $\overline{u}(\pm \infty)=u^\pm$ is a consequence of the following lemma:
\begin{lemma}
\label{closed_boundary}
Let $W:\R^N\to\R_+\cup\{+\infty\}$ be a lower semicontinuous function satisfying {\rm \bf (H1)} and {\rm \bf(H2)}. Then for every map $\sigma \in \dot{H}^1(\R,\R^N)$ such that $E_W(\sigma,\R)<+\infty$ with $E_W$ defined at \eqref{ineg_ener}, there exist two wells $u^-,\, u^+\in\{W=0\}$ such that
$
\lim\limits_{t\to\pm\infty}\sigma(t)=u^\pm.
$
\end{lemma}
\begin{proof}[Proof of Lemma \ref{closed_boundary}]
 We use the fact that the energy bound $E_W(\sigma,\R)<+\infty$ yields a bound on the total variation of $\sigma:\R\to \R^N$ where $\R^N$ is endowed with the pseudo-metric $\mathrm{geod}_W$. More precisely, for every sequence $t_1<\dots< t_k$ in $\R$, we have by Lemma~\ref{regularization_lemma}:
$$
\sum_{i=1}^k \mathrm{geod}_W(\sigma(t_{i+1}),\sigma(t_i))\le \sum_{i=1}^k E_W(\sigma,[t_i,t_{i+1}])\le E_W(\sigma,\R)<+\infty.
$$
In particular, for every $\varepsilon>0$, there exists $R>0$ such that for all $t,s\in\R$ with $t,s\ge R$ or $t,s\le -R$, one has $\mathrm{geod}_W(\sigma(t),\sigma(s))<\varepsilon$. Since by Lemma~\ref{cWgeqGeod}, smallness of $\mathrm{geod}_W(x,y)$ implies smallness of $|x-y|$, we deduce that $\sigma$ has a limit $u^\pm\in\R^N$ at $\pm\infty$. Since $W(\sigma(\cdot))$ is integrable in $\R$, we have furthermore that $W(u^\pm)=0$.
\end{proof}

\bigskip

Now we can prove the convergence of the $x'$-average $\bar u$ at $\pm \infty$ as stated in Theorem \ref{thm1}:

\begin{proof}[Proof of the convergence in $x'$-average in Theorem \ref{thm1}] By Lemma \ref{lemmaV}, we have $E_V(\overline{u},\R)<+\infty$ for the lower semicontinuous function $V:\R^N\to\R_+\cup\{+\infty\}$ satisfying {\rm \bf (H1)} and {\rm \bf(H2)}. By Lemma~\ref{closed_boundary} applied to $E_V$, we deduce that there exists $u^\pm\in\{V=0\}=\{W=0\}$ such that $\lim_{t\to\pm\infty}\overline{u}(t)=u^\pm$.
\end{proof}

The pointwise convergence of $u(x_1, \cdot)$ as $x_1\to \pm\infty$ stated in Theorem \ref{thm1} is proved in the following:

\begin{proof}[Proof of the pointwise convergence in Theorem \ref{thm1}] \label{pagina}
We prove that $u(x_1, \cdot)$ converges a.e. in $\omega$ to $u^\pm\in\{W=0\}$ as $x_1\to \pm \infty$, where $u^\pm$ are the limits $\bar u(\pm \infty)$ of the $x'$-average $\bar u$ proved above. For that, we have by Fubini's theorem:
$$E(u)\geq \int_\Omega |\partial_1 u|^2+W(u)\, \diff x\geq \int_{\omega} E_W(u(\cdot, x'), \R) \, \diff x'$$
with the usual notation
$$E_W(\sigma, \R)=\int_{\R} \abs{\dot \sigma}^2+W(\sigma)\, \diff x_1,\quad \sigma\in \dot{H}^1(\R, \R^N).$$
As $E(u)<\infty$, we deduce that $E_W(u(\cdot, x'), \R)<\infty$ for a.e. $x'\in \omega$.
By Lemma \ref{closed_boundary}, we deduce that for a.e. $x'\in \omega$, there exist two wells $u^\pm(x')$ of $W$
such that
\be
\label{point}
\lim\limits_{x_1\to\pm\infty}u(x_1, x')=u^\pm(x').
\ee
By \eqref{poincare_wirtinger}, as $\bar u(\pm \infty)=u^\pm$, we know that $\| u(R_n^\pm,\cdot)- u^\pm\|_{L^2(\omega,\R^N)}\to 0$ as $n\to\infty$ for two sequences $R_n^\pm\to \pm \infty$. Up to a subsequence, we deduce that
$u(R_n^\pm,\cdot)\to u^\pm$ a.e. in $\omega$ as $n\to \infty$. By \eqref{point}, we conclude that
$u^\pm(x')=u^\pm$ for a.e. $x'\in \omega$.
\end{proof}

\section{The $L^2$ convergence}
\label{sec3}

In this section, we prove that $u(x_1, \cdot)$ converges in $L^2(\omega, \R^N)$ to $u^\pm$ as $x_1\to \pm \infty$. 
The idea is to go beyond the averaging procedure in Section \ref{sec:2} and keep the full information given by the 
$x'$-average energy $e$ introduced at Section \ref{sec:2} over the set ${H}^1(\omega, \R^N)$. More precisely, we extend $e$ to the space $L^2(\omega,\R^N)$  as follows
\be
\label{defe}
e(v)=\begin{cases}
\displaystyle \int_{\omega} \hspace{-0.4cm} - \, \, \Big(|{\nabla'} v|^2+W(v)\Big) \diff x' \quad & \textrm{ if } \, v\in {H}^1(\omega, \R^N),\\
+\infty  \quad & \textrm{ if } \, v\in L^2(\omega, \R^N)\setminus {H}^1(\omega, \R^N).
\end{cases}
\ee
In particular, we have for every $u\in \dot{H}^1(\Omega, \R^N)$,
\be
\label{supl}
E(u)= \int_{\R} \Big(\|\partial_1 u(x_1,\cdot)\|_{L^2(\omega, \R^N)}^2+|\omega| e(u(x_1,\cdot)) \Big)\diff x_1.
\ee
In the sequel, we will also need the following properties of the energy $e$:
\begin{lemma}\label{lemma_e}
If $W:\R^N\to\R_+\cup\{+\infty\}$ is a lower semicontinuous function satisfying {\rm \bf(H2)}, then
\begin{enumerate}
\item
$e$ is lower semicontinuous in $L^2(\omega, \R^N)$, 
\item
the sets of zeros of $e$ and $W$ coincide; moreover $\Sigma:=\{e=0\}=\{W=0\}\subset \R^N$ is compact,
\item
for every \(\ve>0\), we have
\[
k_\ve:=\inf\big\{e(v)\; : \; v\in L^2(\omega,\R^N)\text{ with } d_{L^2}(v,\Sigma)\ge\ve\big\}>0.
\]
\end{enumerate}
\end{lemma}

\begin{proof} We divide the proof in several steps:

\medskip
\noindent \textsc{Step 1. Lower semicontinuity of $e$ in $L^2(\omega, \R^N)$.} Indeed, let $v_n\to v$ in 
$L^2(\omega, \R^N)$. W.l.o.g., we may assume that $(e(v_n))_n$ is bounded, in particular, $(v_n)_n$ is bounded in $H^1(\omega, \R^N)$; thus, $(v_n)_n$ converges to $v$ weakly in $H^1(\omega, \R^N)$. By Step 1 in the proof of Lemma \ref{lemmaV}, we know that $e\big|_{H^1(\omega, \R^N)}$ is lower semicontinuous w.r.t. the weak $H^1$ topology and the conclusion follows. 

\medskip
\noindent \textsc{Step 2. Zeros of $e$.} The equality of the zero sets of $e$ and $W$ is straightforward thanks to the connectedness of $\omega$. Thanks to the assumption 
{\rm \bf(H2)}, the set of zeros $\Sigma$ of $W$ is bounded and by the lower semicontinuity and non-negativity of $W$, the set of zeros $\Sigma$ of $W$ is closed; thus, $\Sigma$ is compact in $\R^N$.

\medskip
\noindent \textsc{Step 3. We prove that $k_\eps>0$.} Assume by contradiction that $k_\eps=0$ for some $\eps>0$. Then there exists a minimizing sequence $v_n\in L^2(\omega, \R^N)$ such that \(d_{L^2}(v_n,\Sigma)\ge\ve\) for every \(n\in\N\) and \(\lim_{n\to\infty}e(v_n)=0\). W.l.o.g., we may assume that $v_n\in {H}^1(\omega, \R^N)$ for every $n$ as $\|v_n\|_{\dot{H^1}}\to 0$. Denoting $\overline{v_n}$ the ($x'$-)average of $v_n$, the Poincar\'e-Wirtinger inequality implies that the sequence $(w_n:=v_n-\overline{v_n})_n$ converges in $H^1(\omega, \R^N)$ to $0$. Up to extracting a subsequence, we may assume that 
$w_n\to 0$ for a.e. $x'\in \omega$.

\medskip
\noindent\emph{Claim:} The sequence $(\overline{v_n})_n$ is bounded in $\R^N$.
\medskip

\noindent Indeed, assume by contradiction that there exists a subsequence of $(\overline{v_n})_n$ (still denoted by $(\overline{v_n})_n$) such that $|\overline{v_n}|\to \infty$ as $n\to \infty$. As $W$ is l.s.c. and $w_n\to 0$ for a.e. $x'\in \omega$, the assumption ${\rm \bf (H2)}$ implies 
$$\liminf_{n\to \infty} W(v_n(x'))=\liminf_{n\to \infty} W(w_n(x')+\overline{v_n})\geq \liminf_{|z|\to \infty} W(z)>0 \quad \textrm{for a.e. $x'\in \omega$}$$
which by integration over $x'\in \omega$ contradicts the assumption $e(v_n)\to 0$. This finishes the proof of the claim.
\medskip

\noindent As a consequence of the claim, we deduce that \((v_n)_{n\in\N}\) is bounded in \(H^1(\omega,\R^N)\). In particular, \((v_n)_{n\in\N}\) has a subsequence that converges in \(L^2(\omega,\R^N)\) to a map \(v\in H^1(\omega,\R^N)\) and we deduce \(d_{L^2}(v,\Sigma)\ge\ve\), in particular, $v$ is not a zero of $e$, i.e., $e(v)>0$. As $e$ is l.s.c. in 
$L^2(\omega, \R^N)$, we have $0=\lim_{n\to\infty}e(v_n)\geq e(v)$, which contradicts that $e(v)>0$. 
\end{proof}

\bigskip

Now we prove the $L^2$-convergence of $u(x_1, \cdot)$ to $u^\pm$ as $x_1\to \pm \infty$:

\begin{proof}[Proof of the $L^2$-convergence in Theorem \ref{thm1}]
Take \(u\in H^1_{loc}(\Omega,\R^N)\) such that \(E(u)<+\infty\) and set \(\sigma(t):=u(t, \cdot)\in H^1(\omega,\R^N)\) for a.e. \(t\in\R\). We prove that \(\sigma(t)\) converges in $L^2(\omega,\R^N)$ to a limit that is a zero in $\Sigma$ as \(t\to+\infty\) (the proof of the convergence as \(t\to-\infty\) is similar). Moreover, we will see that these limits are in fact the zeros $u^\pm$ of $W$ given by the $x'$-average $\bar u$ and the a.e. convergence of $u(x_1, \cdot)$ as $x_1\to \pm \infty$.

\medbreak
\noindent \textsc{Step 1: Continuity.} We prove that \(t\in \R\mapsto\sigma(t)\in L^2(\omega,\R^N)\) is continuous in \(\R\), and moreover, it is a \(\frac12\)-H\"older map. Indeed, for a.e. \(t,s\in\R\), we have
\[
d_{L^2}(\sigma(t),\sigma(s))^2=\int_\omega\Big|\int_t^s\partial_{x_1}u(x_1,x')\diff x_1\Big|^2\diff x'\le\abs{t-s}\|\partial_{x_1}u\|_{L^2(\Omega, \R^N)}^2.
\]

\medbreak
 
\noindent \textsc{Step 2: Convergence of a subsequence $(\sigma(t_n))_n$ to some \(u^+\in\Sigma\).} Since \(e(\sigma(\cdot))\in L^1(\R)\) by \eqref{supl}, there is a sequence \((t_n)_{n\in\N}\to +\infty\) such that \(\lim_{n\to\infty}e(\sigma(t_n))=0\). Exactly like in Step~3 in the proof of Lemma \ref{lemma_e}, we deduce that \((\sigma(t_n))_{n\in\N}\) has a subsequence that converges strongly in \(L^2(\omega,\R^N)\) to some map \(\sigma_\infty\in L^2(\omega,\R^N)\) (the assumption {\rm \bf (H2)} is essential here). Since \(e\) is l.s.c. in \(L^2\) and $e\geq 0$ in $L^2$, we deduce that \(e(\sigma_\infty)=0\) and so, there exists \(u^+\in\Sigma\) such that \(\sigma_\infty\equiv u^+\).

\medbreak

\noindent \textsc{Step 3: Convergence to \(u^+\) in \(L^2\) as \(t\to+\infty\).} Assume by contradiction that 
$\sigma(t)$ does not converge in \(L^2(\omega,\R^N)\) to $u^+$ as $t\to \infty$. Then there is a sequence \((s_n)_{n\in\N}\to +\infty\) such that \(\ve:=\inf_{n\in\N}d_{L^2}(\sigma(s_n),u^+)>0\). Now, by Step 1, the curve \(t\in [s_n,+\infty)\mapsto\sigma(t)\in L^2(\omega,\R^N)\) is continuous. Moreover, $\sigma(s_n)$ doesn't belong to the \(L^2\)-ball centered at \(u^+\) with radius \(\frac{3\ve}{4}\). By Step~2, it has to enter (at some time $t>s_n$) in the \(L^2\)-ball centered at \(u^+\) with radius \(\frac \ve 4\). Therefore, the curve \(\sigma_{\vert (s_n,+\infty)}\) has to cross the ring ${\cal R}:=B_{L^2}(u^+, \frac{3\eps}4)\setminus B_{L^2}(u^+, \frac{\eps}4)$, so it has $L^2$-length larger than $\frac \eps 2$, i.e.,
\[
\int_{ \{t\in (s_n,+\infty)\, :\, \sigma(t)\in {\cal R} \} }  \|\partial_{x_1}u(t,\cdot)\|_{L^2(\omega, \R^N)} \, \diff t=
\int_{ \{t\in (s_n,+\infty)\, :\, \sigma(t)\in {\cal R} \} } \|\dot{\sigma}\|_{L^2(\omega, \R^N)}\, \diff t\ge \frac \eps 2.
\]
Moreover, by the third claim in Lemma \ref{lemma_e}, we know that $e(\sigma(t))\geq k_{\eps/4}$ if $\sigma(t)\in {\cal R}$ (up to lowering $\ve$, we may assume that the other zeros of $\Sigma$ are placed at distance larger than $2\eps$ from $u^+$, the assumption {\rm \bf (H1)} is essential here). We obtain
\begin{align}
\label{antrad}
\int_{s_n}^{+\infty} \sqrt{e(u(t,\cdot))}\;\|\partial_{x_1}u(t,\cdot)\|_{L^2(\omega, \R^N)}\, \diff t&\ge 
\int_{ \{t\in (s_n,+\infty)\, :\, \sigma(t)\in {\cal R} \} } \sqrt{e(u(t,\cdot))}\;\|\partial_{x_1}u(t,\cdot)\|_{L^2(\omega, \R^N)}\, \diff t\\
\nonumber
&\geq \frac{\ve}{2} \sqrt{k_{\ve/4}}.
\end{align}
This is a contradiction with the assumption \(E(u)<+\infty\) implying by \eqref{supl}:
\begin{align*}
2|\omega|^\frac12 \int_{s_n}^{+\infty} \sqrt{e(u(t,\cdot))}\;\|\partial_{x_1}u(t,\cdot)\|_{L^2(\omega, \R^N)}\diff t
&\le  \int_{s_n}^{+\infty} \Big(|\omega| e(u(t,\cdot)) + \|\partial_{x_1}u(t,\cdot)\|_{L^2(\omega, \R^N)}^2\Big)\diff t\\
&\underset{n\to\infty}{\longrightarrow}0.
\end{align*}

\medbreak

\noindent \textsc{Step 4: The \(L^2\) limits $u^\pm$ coincide with the average limits $\bar u(\pm\infty)$.}  This is clear as \(L^2\) convergence implies convergence in average.
\end{proof}

\begin{remark}
i) The above proof does not use (so, it is independent of) the almost everywhere convergence of $u(x_1, \cdot)$ as $x_1\to \pm \infty$ or the convergence of the $x'$-average $\bar u$. Therefore, thanks to this proof, one can obtain as a direct consequence the convergence of the $x'$-average $\bar u$ as well as the almost everywhere convergence of $u(x_1, \cdot)$ as $x_1\to \pm \infty$.\footnote{As the $L^2$-convergence implies almost everywhere convergence of $u(x_1, \cdot)$ only up to a subsequence, one should repeat the argument in the proof of the a.e. convergence in Theorem \ref{thm1} at page \pageref{pagina}.} 

ii) Also, the above proof applies to Lemma \ref{closed_boundary} leading to a second method that does not use the geodesic distance $\mathrm{geod}_W$.

iii) Behind the above proof, the notion of geodesic distance over $L^2(\omega, \R^N)$ with the  degenerate weight $\sqrt{e}$ is hidden (see \eqref{antrad}). Therefore, one could repeat the arguments in the first proof of Theorem \ref{thm1} based on this geodesic distance.  
\end{remark}

The above argument can also be used directly to obtain a second proof for the existence of limits of $\bar u$ at $\pm \infty$ without using the
geodesic pseudo-distance $\mathrm{geod}_W$ (as presented in the proof in Section \ref{sec:2}). For completeness, we redo the proof in the sequel:

\begin{proof}[Second proof of the convergence in $x'$-average in Theorem \ref{thm1}]
Let $u\in \dot{H}^1(\Omega, \R^N)$ such that $E(u)<\infty$. We want to prove that the $x'$-average $\bar u$ admits a limit $u^+$ as $x_1\to \infty$ and $W(u^+)=0$ (the proof of the convergence as \(x_1\to-\infty\) is similar). Let $V$ and $E_V$ given by Lemma \ref{lemmaV}. Recall that $\Sigma:=\{V=0\}=\{W=0\}$ and $E_V(\bar u)\leq \frac1{|\omega|} E(u)<\infty$.

\medskip
\noindent \textsc{Step 1. We prove that for every $\ve>0$,}
\[
\kappa_\ve:=\inf\big\{V(z)\; : \; z\in \R^N, \, d_{\R^N}(z,\Sigma)\ge\ve\big\}>0.
\]
Assume by contradiction that there exists a sequence $(z_n)_n$ such that $V(z_n)\to 0$ and $d_{\R^N}(z_n,\Sigma)\geq \ve$.
By the third claim in Lemma \ref{lemmaV}, we deduce that $(z_n)_n$ is bounded, so that, up to a subsequence, $z_n\to z$ for some $z\in \R^N$ yielding
$d_{\R^N}(z,\Sigma)\ge\ve$ and $V(z)=0$, i.e., $z\in \Sigma$ (since $V$ is l.s.c. and $V\geq 0$) which is a contradiction.

\medskip
\noindent \textsc{Step 2. There exists a sequence $(\bar u(t_n))_n$ converging to a well $u^+\in \Sigma$}. Indeed, as $V(\bar u)\in L^1(\R)$, there exists a sequence $t_n\to \infty$ with $V( \bar u(t_n))\to 0$. By {\rm \bf (H2)}, $(\bar u(t_n))_n$ is bounded, so that up to a subsequence, $\bar u (t_n)\to u^+$ as $n\to \infty$ for some point $u^+\in \R^N$. As $V$ is l.s.c. and $V\geq 0$, we deduce that $V( u^+)=0$, i.e., $u^+\in \Sigma$.

\medbreak

\noindent \textsc{Step 3: Convergence of $\bar u$ to \(u^+\) as \(x_1\to+\infty\).} Assume by contradiction that 
$\bar u(x_1)$ does not converge to $u^+$ as $x_1\to \infty$. Then there is a sequence \((s_n)_{n\in\N}\to +\infty\) such that \(\ve:=\inf_{n\in\N}d_{\R^N}(\bar u(s_n),u^+)>0\). As $\bar u:[s_n,+\infty)\to\R^N$ is continuous, by Step 2, it has to get out of the ball $B(\bar u(s_n), \ve/4)$ and it has to enter in the ball $B(u^+, \ve/4)$. Therefore, $\bar u$ has to cross the ring ${\cal R}:=B(u^+, \frac{3\eps}4)\setminus B(u^+, \frac{\eps}4)\subset \R^N$. Moreover, by Step 1, we know that $V(\bar u(x_1))\geq \kappa_{\eps/4}$ if $\bar u(x_1)\in {\cal R}$ (where we assumed w.l.o.g. that $\eps>0$ is small enough so that the other zeros of $\Sigma$ are placed at distance larger than $2\eps$ from $u^+$). We obtain
$$
\int_{s_n}^{+\infty} \sqrt{V(\bar u(x_1))}\;\big|\frac{\diff}{\diff x_1}\overline{u}(x_1)\big|\, \diff x_1\ge 
\int_{ \{x_1\in (s_n,+\infty)\, :\, \bar u(x_1)\in {\cal R} \} } \sqrt{V(\bar u(x_1))}\; \big|\frac{\diff}{\diff x_1}\overline{u}(x_1)\big|\, \diff x_1\geq \frac{\ve}{2} \sqrt{\kappa_{\ve/4}}.
$$
This is a contradiction with the assumption \(E_V(\bar u)<+\infty\) implying
$$
2\int_{s_n}^{+\infty} \sqrt{V(\bar u(x_1))}\;\big|\frac{\diff}{\diff x_1}\overline{u}(x_1)\big|\diff x_1
\le  \int_{s_n}^{+\infty} \Big(\big|\frac{\diff}{\diff x_1}\overline{u}(x_1)\big|^2+V(\bar u (x_1))\Big)\diff x_1
\underset{n\to\infty}{\longrightarrow}0.
$$

\end{proof}

\section{Proof of Theorem \ref{thm2}}
\label{sec4}

In this section, we consider $d=N$, $\Omega=\R\times \omega$ with $\omega=\mathbb{T}^{d-1}$ and $u\in H^1_{loc}(\Omega, \R^d)$ periodic in $x'\in \omega$ with $\bar u_1=a$ in $\R$ for some constant $a\in \R$ (recall that $\bar u$ is the $x'$-average of $u$). Note that $|\omega|=1$. We set
$$L^2_a(\omega, \R^d):=\left\{v=(v_1, \dots, v_d)\in  L^2(\omega, \R^d)\, :\, \int_{\omega} v_1\, dx'=a\right\}$$ and
$H^1_a(\omega, \R^d):=H^1\cap L^2_a(\omega, \R^d)$. Note that for a.e. $x_1\in \R$, $u(x_1, \cdot)\in 
H^1_a(\omega, \R^d)$.
We define the following energy $e_a$ on the convex closed subset $L^2_a(\omega,\R^d)$ of  $L^2(\omega,\R^d)$:
\be
\label{defea}
e_a(v)=\begin{cases}
\displaystyle \int_{\omega} \Big(|{\nabla'} v|^2+W(v)\Big) \diff x' \quad & \textrm{ if } \, v\in {H}^1_a(\omega, \R^d),\\
+\infty  \quad & \textrm{ if } \, v\in L^2_a(\omega, \R^d)\setminus {H}^1(\omega, \R^d).
\end{cases}
\ee
In particular, we have for every $u\in \dot{H}^1(\Omega, \R^d)$ with $\bar u_1=a$:
\be
\label{supla}
E(u)= \int_{\R} \Big(\|\partial_1 u(x_1,\cdot)\|_{L^2(\omega, \R^d)}^2+ e_a(u(x_1,\cdot)) \Big)\diff x_1.
\ee
The aim is to adapt the proof of Theorem \ref{thm1} given in Section \ref{sec3} to Theorem \ref{thm2}. We start by transfering the properties
 of the energy $e$ in Lemma \ref{lemma_e} to the energy $e_a$ defined in $L^2_a(\omega, \R^d)$. More precisely, 
if $W:\R^d\to\R_+\cup\{+\infty\}$ is a lower semicontinuous function, then
$e_a$ is lower semicontinuous in $L^2_a(\omega, \R^d)$ endowed with the strong $L^2$-norm and the sets of zeros of $e_a$ and $W(a, \cdot)$ coincide, i.e., $$\Sigma^a:=\{v\in L^2_a(\omega, \R^d)\, :\, e_a(v)=0\}=\{z=(a, z')\in \R^d\, :\, W(a, z')=0\}.$$ If in addition $W$ satisfies $\textrm{\bf (H2)}_a$, then
$\Sigma^a$ is compact in $\R^d$ and
for every \(\ve>0\), we have
\[
k^a_{\ve}:=\inf\big\{e_a(v)\; : \; v\in L^2_a(\omega,\R^d)\text{ with } d_{L^2}(v,\Sigma^a)\ge\ve\big\}>0
\]
(the proof of these properties follows by the same arguments presented in the proof of Lemma \ref{lemma_e}).

\begin{proof}[Proof of Theorem \ref{thm2}] 
Let \(u\in H^1_{loc}(\Omega,\R^d)\) such that \(E(u)<+\infty\) and $\bar u_1=a$ in $\R$. We set \(\sigma(t):=u(t, \cdot)\in H^1_a(\omega,\R^d)\) for a.e. \(t\in\R\). We prove that \(\sigma(t)\) converges in $L^2(\omega,\R^d)$ to a limit that is a zero in $\Sigma^a$ as \(t\to+\infty\) (the proof of the convergence as \(t\to-\infty\) is similar). As in Steps~1 and 2 in the proof of the $L^2$-convergence in Theorem \ref{thm1}, we have that
 \(t\in \R\mapsto\sigma(t)\in L^2_a(\omega,\R^d)\) is a \(\frac12\)-H\"older continuous map in $\R$ and there is a sequence \((t_n)_{n\in\N}\to +\infty\) such that \(\sigma(t_n)\to u^+\) in \(L^2(\omega,\R^d)\) for a well \(u^+\in\Sigma^a\) (the assumption 
 $\textrm{\bf (H2)}_a$ is essential here). In order to prove the convergence of $\sigma(t)$ to \(u^+\) in \(L^2\) as \(t\to+\infty\), we argue by contradiction. If $\sigma(t)$ does not converge in \(L^2(\omega,\R^d)\) to $u^+$ as $t\to \infty$, then there is a sequence \((s_n)_{n\in\N}\to +\infty\) such that \(\ve:=\inf_{n\in\N}d_{L^2}(\sigma(s_n),u^+)>0\). We repeat the argument in 
 Step 3 in the proof of the $L^2$-convergence in Theorem \ref{thm1} by restricting ourselves to $L^2_a(\omega,\R^d)$ endowed by the strong $L^2$ topology. More precisely, the continuous curve \(t\in [s_n,+\infty)\mapsto\sigma(t)\in L^2_a(\omega,\R^d)\) has to cross the ring ${\cal R}_a:=\big(B_{L^2}(u^+, \frac{3\eps}4)\setminus B_{L^2}(u^+, \frac{\eps}4)\big)\cap L^2_a(\omega,\R^d)$, so it has $L^2$-length larger than $\frac \eps 2$, i.e.,
\[
\int_{ \{t\in (s_n,+\infty)\, :\, \sigma(t)\in {\cal R}_a \} }  \|\partial_{x_1}u(t,\cdot)\|_{L^2(\omega, \R^d)} \, \diff t=
\int_{ \{t\in (s_n,+\infty)\, :\, \sigma(t)\in {\cal R}_a \} } \|\dot{\sigma}\|_{L^2(\omega, \R^d)}\, \diff t\ge \frac \eps 2.
\]
As $e(\sigma(t))\geq k^a_{\eps/4}$ if $\sigma(t)\in {\cal R}_a$ (up to lowering $\ve$, we may assume that the other zeros of $\Sigma^a$ are placed at distance larger than $2\eps$ from $u^+$, the assumption $\textrm{\bf (H1)}_a$ is essential here), we obtain
\begin{align*}
\int_{ \{t\in (s_n,+\infty)\, :\, \sigma(t)\in {\cal R}_a \} } \sqrt{e_a(u(t,\cdot))}\;\|\partial_{x_1}u(t,\cdot)\|_{L^2(\omega, \R^d)}\, \diff t
\geq \frac{\ve}{2} \sqrt{k^a_{\ve/4}}.
\end{align*}
This is a contradiction with \eqref{supla}:
\begin{align*}
2\int_{s_n}^{+\infty} \sqrt{e_a(u(t,\cdot))}\;\|\partial_{x_1}u(t,\cdot)\|_{L^2(\omega, \R^d)}\diff t
&\le  \int_{s_n}^{+\infty} \Big(e_a(u(t,\cdot)) + \|\partial_{x_1}u(t,\cdot)\|_{L^2}^2\Big)\diff t
\underset{n\to\infty}{\longrightarrow}0.
\end{align*}
Clearly, the \(L^2\) convergence implies also the convergence
in average of $\sigma(t)$ over $\omega$ as $t\to \infty$ as well as the a.e. convergence $\sigma(t)\to u^+$ in $\omega$ but only up to a subsequence. For the full almost everywhere convergence of \(u(x_1,\cdot)\to u^+\), we proceed as follows. First, by the Poincar\'e-Wirtinger inequality on $\omega=\T^{d-1}$, we have for a.e. \(x_1\in \R\),
\[
\int_{\omega}\abs{\nabla' u_1(x_1,x')}^2\diff x'\ge 4\pi^2 \int_{\omega} \abs{u_1(x_1,x')-\bar u_1(x_1)}^2\diff x'=4\pi^2\int_{\omega} \abs{u_1(x_1,x')-a}^2\diff x'.
\]
By Fubini's theorem, we deduce that
\[
E(u)\ge\int_\Omega\big(\abs{\partial_1 u}^2+\abs{\nabla' u_1}^2+ W(u)\big)\diff x
\ge\int_{{{\T^{d-1}}}}E_{W_a}(u(\cdot,x'),\R)\diff x',
\]
where \(W_a(z):=W(z)+4\pi^2\abs{z_1-a}^2\) and, as usual,
$$E_{W_a}(\sigma, \R)=\int_{\R} \big(\abs{\dot \sigma}^2+W_a(\sigma)\big) \diff x_1,\quad \sigma\in \dot{H}^1(\R, \R^N).$$
Hence, $E_{W_a}(u(\cdot, x'), \R)<\infty$ for a.e. $x'\in \omega$.
Note that \(W_a\) is lower semicontinuous and satisfies assumptions $\textrm{\bf (H1)}$ (the set of zeros of \(W_a\) coincides with 
\(\Sigma^a\), which is finite by $\textrm{\bf (H1)}_a$) and the coercivity condition $\textrm{\bf (H2)}$ (thanks to $\textrm{\bf (H2)}_a$). Thus, Lemma 
\ref{closed_boundary} implies that for a.e. $x'\in \omega$, there exist two wells $u^\pm(x')$ of $W_a$ such that
\be
\label{pointa}
\lim\limits_{x_1\to\pm\infty}u(x_1, x')=u^\pm(x').
\ee
By \eqref{poincare_wirtinger}, as $\bar u(\pm \infty)=u^\pm$, we know that $\| u(R_n^\pm,\cdot)- u^\pm\|_{L^2(\omega,\R^N)}\to 0$ as $n\to\infty$ for two sequences $(R_n^\pm)_{n\in\N}\to \pm \infty$. Up to a subsequence, we deduce that
$u(R_n^\pm,\cdot)\to u^\pm$ a.e. in $\omega$ as $n\to \infty$. By \eqref{pointa}, we conclude that
$u^\pm(x')=u^\pm$ for a.e. $x'\in \omega$.
\end{proof}

\paragraph{Acknowledgment.} R.I. acknowledges partial support by the ANR project ANR-14-CE25-0009-01.

\end{document}